\documentclass[10pt]{amsart}
\usepackage{amsmath}
\usepackage{cases}
\usepackage{mathrsfs}
\usepackage{bbm}
\usepackage{amssymb}
\usepackage{amscd}
\usepackage{amsfonts,latexsym,amsmath,amsthm,amsxtra,mathdots,amssymb,latexsym,mathabx}
\usepackage[all,cmtip]{xy}
\RequirePackage{amsmath} \RequirePackage{amssymb}
\usepackage{color}
\usepackage{colordvi}
\usepackage{multicol}
\usepackage{hyperref}
\usepackage{mathtools}
\usepackage[margin=1.1in]{geometry}
\usepackage{xcolor}
\hypersetup{
	colorlinks,
	linkcolor={red!50!black},
	citecolor={blue!50!black},
	urlcolor={blue!80!black}
}

\numberwithin{equation}{section}
\newtheorem{thm}{Theorem}[section]
\newtheorem{lem}[thm]{Lemma}
\newtheorem{cor}[thm]{Corollary}
\newtheorem {rem}[thm]{Remark}
\allowdisplaybreaks[4]

\title{integral moment of the riemann zeta function and hecke $L$ functions}
\author{Zhaoyan Chen}
\address{School of Mathematics, Shandong University\\ Jinan 250100\\China}
\email{zychen@mail.sdu.edu.cn}

\begin{document}
	\begin{abstract}
		Let $f$ be a Hecke cusp form for $SL(2,\mathbb{Z})$. We prove an asymptotic formula for the mixed moment of the product of $\zeta(s)$ and $L(s,f)$ on the critical line. Similarly, we prove an asymptotic formula for the mixed moment of the product of $|\zeta^{2}(s)|$ and $L(s,f)$ on the critical line.
	\end{abstract}


\maketitle
	\section{Introduction}
	In the history of analytic number theory, an important problem is to gain good understanding of the asymptotic formula for the integral mean of the automorphic $L$-function on the critical line. Classically, the first to study is the moments of the Riemann zeta function, that is
	$$
	\begin{aligned}
		I_{k}(T)=\int_{0}^{T}|\zeta(\frac{1}{2}+it)|^{2k} \mathrm{~d} t.
	\end{aligned}
	$$
	Hardy and Littlewood \cite{hardy1916contributions} established their famous result in 1916, providing the asymptotic formula for the first moment: 
	\begin{equation}\label{Hardy}
	\begin{aligned}
		I_{1}(T)\sim T \log T.
	\end{aligned}
\end{equation}
	Subsequently, Ingham \cite{ingham1928mean} got the asymptotic formula for the fourth mean, yielding:
	\begin{equation}\label{Ing}
		\begin{aligned}
			I_{2}(T)\sim\frac{1}{2 \pi^2} T(\log T)^4.
		\end{aligned}
	\end{equation}
    In 1979, Heath-Brown \cite{heath1979fourth} showed that the second moment can be expressed as:
    $$
    \begin{aligned}
    	I_{2}(T)=T\sum_{i=0}^{4}a_{i}(\log T)^{i}+O(T^{7/8+\varepsilon}).
    \end{aligned}
    $$ 
    Keating and Snaith \cite{keating2000random} conjectured that the general form for the $k$-th moment is given by:
	$$
	\begin{aligned}
		I_{k}(T)=T \mathcal{P}_k(\log T)+O\left(T^{1 / 2+\varepsilon}\right),
	\end{aligned}
	$$
	where $\mathcal{P}_k$ is a polynomial of degree $k^2$. The above conjecture has been proved for $k=1$ \cite{ingham1928mean} and $k=2$. When $k=2$, the error term $O(T^{1/2+\varepsilon})$ has been obtained in the case of a smooth weight function \cite{ivic1995fourth}. The current record for the error term in \eqref{Hardy} was due to Bourgain and Watt \cite{bourgain2018decoupling} who showed that $O(T^{\frac{1515}{4816}+\varepsilon})$. Zavorotny\u{\i} \cite{MR1683661} improved the error term in \eqref{Ing} to $O(T^{2/3+\varepsilon})$. The best known result to date, as presented in \cite{ivic1995fourth2}, is the following asymptotic formula:
	\begin{equation}\label{Ivic}
		\begin{aligned}
			I_{2}(T)=T\mathcal{P}_k(\log T)+O(T^{2/ 3}(\log T)^8).
		\end{aligned}
	\end{equation}
	Despite many attempts, the asymptotic formulas for higher moments have not been completely solved. 
	
	Due to good analytical properties of holomorphic cusp forms, many mathematicaians have also calculated the integral mean of the Hecke $L$-functions. For instance, Potter \cite{potter1940mean} obtained an asymptotic formula for the second moment when $\operatorname{Re}(s)>1/2$. In 1982, Good \cite{good1982square} extended this result to the critical line using the spectral theory of automorphic forms on $GL(2)$. The following asymptotic holds:
	\begin{equation}\label{good}
		\begin{aligned}
			\int_0^T\left|L\left(\frac{1}{2}+i t,f\right)\right|^2 d t=2 c_{-1} T\left\{\log \frac{T}{2 \pi e}+c_0\right\}+O\left((T \log T)^{2 / 3}\right).
		\end{aligned}
	\end{equation}
	Then Meurman \cite{meurman1987order} generalized the above result to the Maass case. Similarly, no effective asymptotic formula has been obtained for the case of higher moments. 
	
	Das and Khan \cite{das2015simultaneous} discussed the $q$-aspect case and proved the following asymptotic formula:
	$$
	\begin{aligned}
	\sideset{}{^*}\sum_{\substack{\chi \bmod q \\ \chi(-1)=1}} L\left(\frac{1}{2}, f \otimes \chi\right) \overline{L\left(\frac{1}{2}, \chi\right)}=\frac{q-2}{2} L(1, f)+O(q^{7/8+\theta+\varepsilon}),
	\end{aligned}
    $$
    where $f$ is a Hecke-Maass form for $SL(2,\mathbb{Z})$, $q$ is a prime and $\theta=7/64$. Motivated by \cite{das2015simultaneous}, the aim of this paper is to study the mean integral of the Hecke $L$-function and the Riemann zeta function over the central line. We will prove the following results.
	
	\begin{thm}\label{t1} Let $f$ be a Hecke cusp form (holomorphic or Maass) for $SL(2,\mathbb{Z})$. Suppose $V(x)$ is a smooth function with support contained in $[1,2]$, further satisfying the condition
		$$
		\begin{aligned}
			V^{(i)}(x)\ll\Delta^{i},
		\end{aligned}
		$$
		for $i\geq0$ with $\Delta\leq T^{1/2-\varepsilon}$. For any $\varepsilon>0$, then we have that
		\begin{equation}\label{eqt1}
			\begin{aligned}
				\int_{\mathbb{R}}V\left(\frac{t}{T}\right)L\left(\frac{1}{2}+it, f\right)\left|\zeta\left(\frac{1}{2}-it\right)\right|^{2}\mathrm{~d} t=\frac{c}{2}L(1,f)T\log T+c_{f}T +O(T^{\frac{1}{2}+\varepsilon}),
			\end{aligned}
		\end{equation}
	    where $c=\int_{\mathbb{R}}V(\xi)\mathrm{~d} \xi$ and the constant $c_{f}$ is given by
	    $$
	    \begin{aligned}
	    	c_{f}=\frac{1}{2\pi i}\int_{\mathbb{R}}V(\xi)\mathrm{~d} \xi\int_{(\varepsilon)}\frac{L(1-w,f)}{-w^2}e^{2w^2+i\pi   w/2} \mathrm{~d} w+L(1,f)\int_{\mathbb{R}}V(\xi)\left[\frac{i\pi}{4}+\frac{1}{2}\log\left(\frac{\xi}{2\pi}\right)\right]\mathrm{~d} \xi.
	    \end{aligned}
	    $$		
	    \end{thm}
        \begin{rem} With the notation as above, it is not difficult to get
		\begin{equation}
			\begin{aligned}
				\int_{\mathbb{R}}V\left(\frac{t}{T}\right)L\left(\frac{1}{2}+it, f\right)\zeta\left(\frac{1}{2}-it\right)\mathrm{~d} t=cTL(1,\,f)+O(T^{\frac{1}{2}+\varepsilon}),
			\end{aligned}
		\end{equation}
	where $c=\int_{\mathbb{R}}V(\xi)\mathrm{~d} \xi$.	
	    \end{rem}
    \begin{rem}
    The same method applies for $f$ either a holomorphic or a Maass form. In fact, the case where $f$ is holomorphic is simpler, so we will present the proof for the Maass case.
    \end{rem}
	Since the test function $V$ in the statement allows oscillations, we can derive a nontrivial asymptotic formula for the integral with a sharp cut. Indeed, we have the following results:
	
	\begin{cor}\label{c3} For any $\varepsilon>0$, we have
	$$
	\begin{aligned}
		\int_{T}^{2T}L\left(\frac{1}{2}+it, f\right)\left|\zeta\left(\frac{1}{2}-it\right)\right|^{2}\mathrm{~d} t=\frac{c}{2}L(1,f)T\log T+c_{f}T+O(T^{\frac{2}{3}+\varepsilon}).
	\end{aligned}
	$$
    \end{cor}
    \begin{cor}\label{c4} For any $\varepsilon>0$, we have
    	$$
    	\begin{aligned}
    		\int_{T}^{2T}L\left(\frac{1}{2}+it, f\right)\zeta\left(\frac{1}{2}-it\right)\mathrm{~d} t=cTL(1,\,f)+O(T^{\frac{7}{12}+\varepsilon}).
    	\end{aligned}
    	$$
    \end{cor}
	    In this paper, $\varepsilon$ is an arbitrarily small positive constant which is not necessarily
		the same at each occurrence. 
	
	\section{Preliminaries}
	\subsection{Approximate functional equations} 
	
	Let $L(s,f)$ be an $L$-function of degree $d$ in the sense of \cite{iwaniec2021analytic}. More precisely, we have
	$$
	L(s, f)=\sum_{n \geq 1} \frac{\lambda_f(n)}{n^s}=\prod_p \prod_{j=1}^d\left(1-\frac{\alpha_j(p)}{p^s}\right)^{-1}
	$$
	with $\lambda_f(1)=1, \lambda_f(n) \in \mathbb{C}, \alpha_j(p) \in \mathbb{C}$, such that the series and Euler products are absolutely convergent for $\operatorname{Re}(s)>1$. There exist an integer $q(f) \geq 1$ and a gamma factor
	
	$$
	\gamma(s, f)=\pi^{-d s / 2} \prod_{j=1}^d \Gamma\left(\frac{s-\kappa_j}{2}\right)
	$$
	with $\kappa_j \in \mathbb{C}$ and $\operatorname{Re}\left(\kappa_j\right)<1 / 2$ such that the complete $L$-function
	$$
	\Lambda(s, f)=q(f)^{s / 2} \gamma(s, f) L(s, f)
	$$
	admits an analytic continuation to a meromorphic function for $s \in \mathbb{C}$ of order 1 with poles at most at $s=0$ and $s=1$. Moreover we assume that $L(s, f)$ satisfies the functional equation
	$$
	\Lambda(s, f)=\varepsilon(f) \Lambda(1-s, \bar{f}).
	$$
	The following lemma can be derived from \cite[Theorem 5.3]{iwaniec2021analytic}. Below we simplify the above notation by not displaying the dependence on $f$ and we write $q=q(f)$. 
	\begin{lem} With notation as above, we have
		\begin{equation}\label{app}
			\begin{aligned}
				L(1/2+i t,f)=\sum_{n\geq1}\frac{\lambda_{f}(n)}{n^{1/2+it}}W_{1/2+it}\left(\frac{n}{\sqrt{q}}\right)+\varepsilon(f){q}^{1/2-s}\frac{\gamma(1/2-it,f)}{\gamma(1/2+it,f)}\sum_{n\geq1}\frac{\lambda_{f}(n)}{n^{1/2-it}}W_{1/2-it}\left(\frac{n}{\sqrt{q}}\right),
			\end{aligned}
		\end{equation}
		where for $x, c>0$, $W_{s}(x)$ is a smooth function defined by 
		$$
		\begin{aligned}
			W_{s}(x)=\frac{1}{2 \pi i}\int_{(c)}x^{-u}G(u)\frac{\gamma(s+u,f)}{\gamma(s,f)}\frac{\mathrm{~d} u}{u}
		\end{aligned}
		$$
		and $G(u)$ be any function which is holomorphic and bounded in the strip $-4<\operatorname{Re}(u)<4$, even, and normalized by $G(0)=1$. 
	\end{lem}
    \begin{rem}
    We have that $W_{s}(n)\ll\left(\frac{n}{t^{d/2}}\right)^{-A}$ for any $A>0$, thus the sums in \eqref{app} are essentially supported on $n<T^{d/2+\varepsilon}$.
    \end{rem}
	\begin{lem}\label{2} For $|t|\in[T,2T]$, we have
		\begin{equation}\label{huangapp}
			\begin{aligned}
				L(1/2+i t, f)&=\frac{1}{2 \pi i} \int_{\varepsilon-i T^{\varepsilon}}^{\varepsilon+i T^{\varepsilon}} \sum_{\substack{N \leq T^{d/2+\varepsilon} \\
						N\,d y a d i c}} \sum_{n \geq 1} \frac{\lambda_f(n)}{n^{1 / 2+i t+w}} V_1\left(\frac{n}{N}\right)\left(\frac{ t}{2\pi}\right)^{dw/2} \times e^{i\operatorname{sgn}(t) \pi dw / 4}q^{w/2} \frac{G(w)}{w} \mathrm{~d} w \\
				& +\frac{\varepsilon(f)}{2 \pi i} \int_{\varepsilon-i T^\varepsilon}^{\varepsilon+i T^\varepsilon} \sum_{\substack{N \leq T^{d/2+\varepsilon} \\
						N\,d y a d i c}} \sum_{n \geq 1} \frac{\lambda_f(n)}{n^{1 / 2-i t+w}} V_1\left(\frac{n}{N}\right)\left(\frac{t}{2 \pi e}\right)^{-dit}\left(\frac{t}{2 \pi}\right)^{dw/2} \\
				& \times e^{i\pi\operatorname{sgn}(t)(-  dw / 4+ d/4+\sum_{j}\kappa_{j}/2)}q^{w/2-it} \frac{G(w)}{w} \mathrm{~d} w+O\left(T^{d/4-1+\varepsilon}\right), \\
			\end{aligned}
		\end{equation}
		where $V_1$ is a fixed smooth function with $\operatorname{supp}V\subset[1,2]$. 
	\end{lem}
	\begin{proof} The calculation is the same as Huang \cite[Lemma 11]{huang2021rankin}, but the results are slightly different. As $|t| \rightarrow \infty$, Stirling's formula gives
		$$
		\Gamma(\sigma+i t)=\sqrt{2 \pi}|t|^{\sigma-\frac{1}{2}+i t} \exp \left(-\frac{\pi}{2}|t|-i t+i \operatorname{sgn}(t) \frac{\pi}{2}\left(\sigma-\frac{1}{2}\right)\right)\left(1+O\left(|t|^{-1}\right)\right) \text {. }
		$$
		
		\noindent Hence for $|t| \in[T, 2 T]$ with $T$ large, and $\varepsilon \leq \operatorname{Re}(w) \ll 1,|\operatorname{Im}(w)| \ll T^{\varepsilon}$, and $\kappa_j \ll 1$, we have
		\begin{equation}\label{stir1}
			\frac{\Gamma\left(\frac{1 / 2+i t+w-\kappa_j}{2}\right)}{\Gamma\left(\frac{1 / 2+i t-\kappa_j}{2}\right)}=\left(\frac{|t|}{2}\right)^{w / 2} e^{i\operatorname{sgn}(t)\pi w / 4}\left(1+O\left(T^{ \varepsilon-1}\right)\right),
		\end{equation}
		\begin{equation}\label{stir2}
			\frac{\Gamma\left(\frac{1 / 2-i t-\kappa_j}{2}\right)}{\Gamma\left(\frac{1 / 2+i t-\kappa_j}{2}\right)}=\left(\frac{|t|}{2 e}\right)^{-i t}e^{i\pi \operatorname{sgn}(t)(1/4+\kappa_{j}/2
				)}\left(1+O\left(T^{-1}\right)\right) .
		\end{equation}
		
		\noindent For $|t| \in[T, 2 T]$, by the approximate functional equation \eqref{app} we have
		$$
		\begin{aligned}
			L(1 / 2+i t, f)= & \frac{1}{2 \pi i} \sum_{n \geq 1} \frac{\lambda_f(n)}{n^{1 / 2+i t}}  \int_{(2)} \frac{\gamma(1 / 2+i t+w, f)}{\gamma(1 / 2+i t, f)} \frac{q^{w/2}}{n^w} \frac{G(w)}{w} \mathrm{~d} w \\
			& +\frac{\varepsilon(f)q^{-it}}{2 \pi i} \sum_{n \geq 1} \frac{\overline{\lambda_f(n)}}{n^{1 / 2-i t}}  \int_{(2)} \frac{\gamma( 1 / 2-i t+w,f)}{\gamma( 1 / 2+i t,f)} \frac{q^{w/2}}{n^w} \frac{G(w)}{w} \mathrm{~d} w \\
			& +O\left(T^{-A}\right),
		\end{aligned}
		$$
		
		\noindent where $G(w)=e^{w^2}$. Shifting the lines of integration to $\operatorname{Re}(w)=\varepsilon$, no poles are encountered in this shifting. By \cite[Proposition 5.4]{iwaniec2021analytic} we have
		$$
		\begin{aligned}
			\sum_{n>T^{d/2+\varepsilon}}\frac{\lambda_{f}(n)}{n^{1/2+it}}W_{1/2+it}(n)\ll\sum_{n>T^{d/2+\varepsilon}}\left|\frac{\lambda_{f}(n)}{n^{1/2+it  }}\right|\left(\frac{n}{t^{d/2}}\right)^{-A}\ll T^{-A}.
		\end{aligned}
		$$
		Hence  we can truncate the $n$-sum at $n \leq T^{d/2+\varepsilon}$ for the first sum and at $n \leq T^{d/2+\varepsilon}$ for the second sum above with a negligible error. By a simple calculation we have
		$$
		\begin{aligned}
			\sum_{n\leq T^{d/2+\varepsilon}}\frac{\lambda_{f}(n)}{n^{1/2+it+\varepsilon}}\int_{\varepsilon+i T^{\varepsilon}}^{\varepsilon+i\infty}\frac{\gamma(1 / 2+i t+w, f)}{\gamma(1 / 2+i t, f)} \frac{1}{n^w} \frac{G(w)}{w} \mathrm{~d} w\ll \int_{ T^{\varepsilon}}^{\infty}e^{-u^{2}}\mathrm{~d}u\ll e^{-T^{2\varepsilon}}.
		\end{aligned}
		$$
		Hence we can truncate at $|\operatorname{Im}(w)| \leq T^{\varepsilon}$ with a negligible error term. By Stirling's formula \eqref{stir1} and \eqref{stir2}, and then by a smooth partition of unity, we proved \eqref{huangapp}.
	\end{proof}
    
    In this paper we give the proof of Theorem \ref{t1} for even Hecke-Maass forms, the details for odd forms being entirely similar. Thus throughtout, $f$ will denote an even Hecke-Maass cusp form for the full modular group with Laplacian eigenvalue $\frac{1}{4}+T_{f}^{2}$, where $T_{f}$ is real. Let $\lambda_{f}(n)$ denote the eigenvalue of the $n$-th Hecke operator corresponding to $f$. For the even Hecke-Maass cusp form, we define the $L$-function
    \begin{equation*}
    	L(s,f)=\sum_{n \geq 1} \lambda_f(n)n^{-s}.
    \end{equation*}
    The above function means that $d=2$, $\lambda_{f}(n)=\overline{\lambda_{f}(n)}$ and $q=\varepsilon(f)=1$.\\

    Let the $L$-function in \eqref{huangapp} be $\zeta(s)$, and for $t\in[T,2T]$ we have that
    \begin{equation}\label{zetaapp1}
    	\begin{aligned}
    		\zeta(1 / 2+i t)&=\frac{1}{2 \pi i} \int_{\varepsilon-i T^{\varepsilon}}^{\varepsilon+i T^{\varepsilon}} \sum_{\substack{K \leq T^{1/ 2+\varepsilon} \\
    				K\,d y a d i c}} \sum_{k \geq 1} \frac{1}{k^{1 / 2+i t+w}} V_1\left(\frac{k}{K}\right)\left(\frac{ t}{2\pi}\right)^{ w / 2} \times e^{i \pi  w / 4} \frac{G(w)}{w} \mathrm{~d} w \\
    		& +\frac{1}{2 \pi i} \int_{\varepsilon-i T^\varepsilon}^{\varepsilon+i T^\varepsilon} \sum_{\substack{K \leq T^{1/ 2+\varepsilon} \\
    				K\,d y a d i c}} \sum_{k \geq 1} \frac{1}{k^{1 / 2-i t+w}} V_1\left(\frac{k}{K}\right)\left(\frac{t}{2 \pi e}\right)^{-i t}\left(\frac{t}{2 \pi}\right)^{ w / 2} \\
    		& \times e^{-i \pi  (w / 4-1/4)}  \frac{G(w)}{w} \mathrm{~d} w+O\left(T^{-3/ 4+\varepsilon}\right),\\
    	\end{aligned}
    \end{equation}
    and then
    \begin{equation}\label{zetaapp2}
    	\begin{aligned}
    		\zeta(1 / 2-i t)&=\frac{1}{2 \pi i} \int_{\varepsilon-i T^{\varepsilon}}^{\varepsilon+i T^{\varepsilon}} \sum_{\substack{M \leq T^{1/ 2+\varepsilon} \\
    				M\,d y a d i c}} \sum_{m \geq 1} \frac{1}{m^{1 / 2-i t+w}} V_1\left(\frac{m}{M}\right)\left(\frac{ t}{2\pi}\right)^{ w / 2} \times e^{- i \pi  w / 4} \frac{G(w)}{w} \mathrm{~d} w \\
    		& +\frac{1}{2 \pi i} \int_{\varepsilon-i T^\varepsilon}^{\varepsilon+i T^\varepsilon} \sum_{\substack{M \leq T^{1/ 2+\varepsilon} \\
    				M\,d y a d i c}} \sum_{m \geq 1} \frac{1}{m^{1 / 2+i t+w}} V_1\left(\frac{m}{M}\right)\left(\frac{t}{2 \pi e}\right)^{i t}\left(\frac{t}{2 \pi}\right)^{ w / 2} \\
    		& \times e^{i \pi  (w / 4-1/4)}  \frac{G(w)}{w} \mathrm{~d} w+O\left(T^{-3/ 4+\varepsilon}\right).\\
    	\end{aligned}
    \end{equation}
  
	\subsection{Sums of Fourier coefficients}
	The Ramanujan conjecture for the Fourier coefficients of $f$ is known on average. By the Rankin-Selberg theory, we have
	\begin{lem}\label{3} 
		$$
		\begin{aligned}
			\sum_{n\leq x}|\lambda_{f}(n)|^{2}\ll x .
		\end{aligned}
		$$
	\end{lem}

    \noindent The following result (see \cite[Theorem 8.1]{iwaniec2021spectral}) shows that Fourier coefficients are orthogonal to additive characters on average.
	\begin{lem}\label{5} For any real number $\alpha$ and $\varepsilon>0$, we have
		$$
		\begin{aligned}
			\sum_{n\leq x}\lambda_{f}(n)e(\alpha n)\ll x^{1/2+\varepsilon}.
		\end{aligned}
		$$
	\end{lem}
	\subsection{Estimation of integrals}
	According to \cite{kiral2019oscillatory}. Let $\mathcal{F}$ be an index set and $X=X_T: \mathcal{F} \rightarrow \mathbb{R}_{\geq 1}$ be a function of $T \in \mathcal{F}$. A family of $\left\{w_T\right\}_{T \in \mathcal{F}}$ of smooth functions supported on a product of dyadic intervals in $\mathbb{R}_{>0}^d$ is called $X$-inert if for each $j=\left(j_1, \ldots, j_d\right) \in \mathbb{Z}_{\geq 0}^d$ we have
	
	$$
	\sup _{T \in \mathcal{F}} \sup _{\left(x_1, \ldots, x_d\right) \in \mathbb{R}_{>0}^d} X_T^{-j_1-\cdots-j_d}\left|x_1^{j_1} \cdots x_d^{j_d} w_T^{\left(j_1, \ldots, j_d\right)}\left(x_1, \ldots, x_d\right)\right| \ll_{j_1, \ldots, j_d} 1
	$$
	We will use the following stationary phase lemma several times.
	
	\begin{lem}\label{6} Suppose $w=w_T$ is a family of $X$-inert in $\xi$ with compact support on $[Z, 2 Z]$, so that $w^{(j)}(\xi) \ll(Z / X)^{-j}$. Suppose that on the support of $w, h=h_T$ is smooth and satisfies that $h^{(j)}(\xi) \ll \frac{Y}{Z}$, for all $j \geq 0$. Let
		
		$$
		I=\int_{\mathbb{R}} w(\xi) e^{i h(\xi)} \mathrm{d} \xi.
		$$
		(i) If $h^{\prime}(\xi) \gg \frac{Y}{Z}$ for all $\xi \in \operatorname{supp} w$. Suppose $Y / X \geq 1$. Then $I \ll_A Z(Y / X)^{-A}$ for $A$ arbitrarily large.\\
		(ii) If $h^{\prime \prime}(\xi) \gg \frac{Y}{Z^2}$ for all $\xi \in \operatorname{supp} w$, and there exists $\xi_0 \in \mathbb{R}$ such that $h^{\prime}\left(\xi_0\right)=0$. Suppose that $Y / X^2 \geq 1$. Then we have
		
		$$
		I=\frac{e^{i h\left(\xi_0\right)}}{\sqrt{h^{\prime \prime}\left(\xi_0\right)}} W_T\left(\xi_0\right)+O_A\left(Z\left(Y / X^2\right)^{-A}\right)
		$$
		for any $A>0$, for some $X$-inert family of functions $W_T$ (depending on $A$ ) supported on $\xi_0 \asymp Z$.
	\end{lem}
\begin{proof}
	See \cite[Lemma 3.1]{kiral2019oscillatory}.
\end{proof}
	\begin{lem}\label{7} For any $T>0$ and any sequence of complex numbers $a_n$, we have 
		$$
		\begin{aligned}
			\int_0^T{\left|\sum_{n \leq N} a_n n^{i t}\right|} ^2 \mathrm{~d} t \ll(T+N) \sum_{n \leq N}\left|a_n\right|^2.
		\end{aligned}
		$$
	\end{lem}
	\begin{proof}
	See \cite[Theorem 9.1]{iwaniec2021analytic}. 
	\end{proof}
	\section{Proof of Theorem \ref{t1}}
	
	Applying the approximate functional equations \eqref{huangapp}, \eqref{zetaapp1} and \eqref{zetaapp2}, we have that the left hand side of \eqref{eqt1} equals 
	\begin{align}\label{eight}
	&\frac{1}{(2 \pi i)^3}\int_{\mathbb{R}}V\left(\frac{t}{T}\right)\Bigg(\int_{\varepsilon-i T^{\varepsilon}}^{\varepsilon+i T^{\varepsilon}} \sum_{\substack{N \leq T^{1+\varepsilon} \\
			N\,d y a d i c}} \sum_{n \geq 1} \frac{\lambda_f(n)}{n^{1 / 2+i t+w}} V_1\left(\frac{n}{N}\right)\left(\frac{ t}{2\pi}\right)^{w} \times e^{i\pi w / 2}\frac{G(w)}{w} \mathrm{~d} w\\
	&+\int_{\varepsilon-i T^\varepsilon}^{\varepsilon+i T^\varepsilon} \sum_{\substack{N \leq T^{1+\varepsilon} \nonumber\\
			N\,d y a d i c}} \sum_{n \geq 1} \frac{\lambda_f(n)}{n^{1 / 2-i t+w}} V_1\left(\frac{n}{N}\right)\left(\frac{t}{2 \pi e}\right)^{-2it}\left(\frac{t}{2 \pi}\right)^{w}
	e^{i\pi(-  w / 2+ 1/2)} \frac{G(w)}{w} \mathrm{~d} w+O\left(T^{-1/ 2+\varepsilon}\right) \Bigg)\nonumber\\
	&\times\Bigg(\int_{\varepsilon-i T^{\varepsilon}}^{\varepsilon+i T^{\varepsilon}} \sum_{\substack{M \leq T^{1/ 2+\varepsilon} \\
			M\,d y a d i c}} \sum_{m \geq 1} \frac{1}{m^{1 / 2-i t+w}} V_1\left(\frac{m}{M}\right)\left(\frac{ t}{2\pi}\right)^{ w / 2} \times e^{- i \pi  w / 4} \frac{G(w)}{w} \mathrm{~d} w \nonumber\\
	& +\int_{\varepsilon-i T^\varepsilon}^{\varepsilon+i T^\varepsilon} \sum_{\substack{M \leq T^{1/ 2+\varepsilon} \\
			M\,d y a d i c}} \sum_{m \geq 1} \frac{1}{m^{1 / 2+i t+w}} V_1\left(\frac{m}{M}\right)\left(\frac{t}{2 \pi e}\right)^{i t}\left(\frac{t}{2 \pi}\right)^{ w / 2}
	e^{i \pi( w / 4-1/4)}  \frac{G(w)}{w} \mathrm{~d} w+O\left(T^{-3/ 4+\varepsilon}\right)\Bigg)\nonumber\\
	&\times\Bigg(\int_{\varepsilon-i T^{\varepsilon}}^{\varepsilon+i T^{\varepsilon}} \sum_{\substack{K \leq T^{1/ 2+\varepsilon} \\
			K\,d y a d i c}} \sum_{k \geq 1} \frac{1}{k^{1 / 2+i t+w}} V_1\left(\frac{k}{K}\right)\left(\frac{ t}{2\pi}\right)^{ w / 2} \times e^{i \pi  w / 4} \frac{G(w)}{w} \mathrm{~d} w \nonumber\\
	&+\int_{\varepsilon-i T^\varepsilon}^{\varepsilon+i T^\varepsilon} \sum_{\substack{K \leq T^{1/ 2+\varepsilon} \\
			K\,d y a d i c}} \sum_{k \geq 1} \frac{1}{k^{1 / 2-i t+w}} V_1\left(\frac{k}{K}\right)\left(\frac{t}{2 \pi e}\right)^{-i t}\left(\frac{t}{2 \pi}\right)^{ w / 2}
	e^{-i \pi  (w / 4-1/4)}  \frac{G(w)}{w} \mathrm{~d} w+O\left(T^{-3/ 4+\varepsilon}\right)\Bigg)\mathrm{~d} t.\nonumber
	\end{align}
	For convenience, we'll write \eqref{eight} as
	$$
	\begin{aligned}
	\frac{1}{(2 \pi i)^3}\int_{\mathbb{R}}V\left(\frac{t}{T}\right)\left[\left(S_1+S_2\right)\left(S_3+S_4\right)\left(S_5+S_6\right)+S_7\right]\mathrm{~d} t,
	\end{aligned}
    $$
    where $S_7$ comes from the error term of \eqref{eight}. Multiplying out the summand above leads to several cross terms, which we will analyze one by one.
	
	\noindent \textbf{Lemma 4.1} Given any $\varepsilon>0$, we have
        $$
		\begin{aligned}
		I_1\coloneq\frac{1}{(2 \pi i)^3}\int_{\mathbb{R}}V\left(\frac{t}{T}\right)S_{1}S_{3}S_{5}\mathrm{~d} t=\frac{c}{2}L(1,f)T\log T+c_{f}T +O(T^{\frac{1}{2}+\varepsilon}),
	\end{aligned}
    $$
    where $c$ and $c_{f}$ are constants as defined in Theorem \ref{t1}.
	\begin{proof} 
		We can exchange the order of integral and summation, rewrite $I_1$ as
		$$
		\begin{aligned}
			&\frac{1}{(2\pi i)^3}\int_{\varepsilon-iT^\varepsilon}^{\varepsilon+iT^\varepsilon}\int_{\varepsilon-iT^\varepsilon}^{\varepsilon+iT^\varepsilon}\int_{\varepsilon-iT^\varepsilon}^{\varepsilon+iT^\varepsilon}\sum_{\substack{N \leq T^{1+\varepsilon} \\
					N\,d y a d i c}} \sum_{n \geq 1}  \frac{\lambda_{f}(n)}{n^{1/2+u}}V_1\left(\frac{n}{N}\right)\sum_{\substack{M \leq T^{1/2+\varepsilon} \\
					M\,d y a d i c}} \sum_{m\geq 1} \frac{1}{m^{1/2+v}}V_1\left(\frac{m}{M}\right)\sum_{\substack{K \leq T^{1/2+\varepsilon} \\
					K\,d y a d i c}} \sum_{k\geq 1}\\&\times \frac{1}{k^{1/2+w}}V_1\left(\frac{k}{K}\right){\left(\frac{1}{2\pi}\right)}^{\frac{2u+v+w}{2}}e^{\frac{i\pi(2u-v+w)}{4}}\frac{G(u)G(v)G(w)}{uvw}\int_{\mathbb{R}}V\left(\frac{t}{T }\right)t^{\frac{2u+v+w}{2}}e^{it\log\frac{m}{nk}}\mathrm{~d} t\mathrm{~d} u\mathrm{~d} v\mathrm{~d} w.
		\end{aligned}
		$$
		We first deal with the $t$-integral. It is easy to see that the main term in $I_1$ consists of those terms with $m=nk$. Making a change of variable $t=T\xi$, we get
		\begin{equation}\label{main}
			\begin{aligned}
				\frac{T}{(2\pi i)^3}&\int_{\varepsilon-iT^\varepsilon}^{\varepsilon+iT^\varepsilon}\int_{\varepsilon-iT^\varepsilon}^{\varepsilon+iT^\varepsilon}\int_{\varepsilon-iT^\varepsilon}^{\varepsilon+iT^\varepsilon}\int_{\mathbb{R}}\sum_{\substack{N \leq T^{1/2+\varepsilon} \\
						N\,d y a d i c}} \sum_{n \geq 1}  \frac{\lambda_{f}(n)}{n^{1+u+v}}V_1\left(\frac{n}{N}\right)\sum_{\substack{K \leq T^{1/2+\varepsilon} \\
						K\,d y a d i c}} \sum_{k\geq 1}\frac{1}{k^{1+v+w}}V_1\left(\frac{k}{K}\right)\\&\times V(\xi)  {\left(\frac{T\xi}{2\pi}\right)}^{\frac{2u+v+w}{2}}e^{\frac{i\pi(2u-v+w)}{4}}\frac{G(u)G(v)G(w)}{uvw}\mathrm{~d} \xi\mathrm{~d} u\mathrm{~d} v\mathrm{~d} w.
			\end{aligned}
		\end{equation}
	    First, shifting the lines of integration to $\operatorname{Re}(u)=-1/2+\varepsilon$, we pick up the residue of
	    $$
	    \begin{aligned}
	    	\frac{T}{(2\pi i)^2}&\int_{\varepsilon-iT^\varepsilon}^{\varepsilon+iT^\varepsilon}\int_{\varepsilon-iT^\varepsilon}^{\varepsilon+iT^\varepsilon}\int_{\mathbb{R}}\sum_{\substack{N \leq T^{1/2+\varepsilon} \\
	    			N\,d y a d i c}} \sum_{n \geq 1}  \frac{\lambda_{f}(n)}{n^{1+v}}V_1\left(\frac{n}{N}\right)\sum_{\substack{K \leq T^{1/2+\varepsilon} \\
	    			K\,d y a d i c}} \sum_{k\geq 1}\frac{1}{k^{1+v+w}}V_1\left(\frac{k}{K}\right)\\&\times V(\xi){\left(\frac{T\xi}{2\pi}\right)}^{\frac{v+w}{2}}e^{\frac{i\pi(-v+w)}{4}}\frac{G(v)G(w)}{vw}\mathrm{~d} \xi\mathrm{~d} v\mathrm{~d} w
	    \end{aligned}
	    $$
	    at $u=0$. The integrals on $\operatorname{Re}(u)=-1/2+\varepsilon$, $\operatorname{Re}(v)=\operatorname{Re}(w)=\varepsilon$ can be bounded by
	    \begin{equation}\label{bound}
	    	\begin{aligned}
	    		\Bigg|T^{1/2+\varepsilon}\sum_{\substack{N \leq T^{1/2+\varepsilon} \\
	    				N\,d y a d i c}} \sum_{n \geq 1}  \frac{\lambda_{f}(n)}{n^{1/2+\varepsilon}}V_1\left(\frac{n}{N}\right)\sum_{\substack{K \leq T^{1/2+\varepsilon} \\
	    				K\,d y a d i c}} \sum_{k\geq 1}\frac{1}{k^{1+\varepsilon}}V_1\left(\frac{k}{K}\right)\Bigg|.
	    	\end{aligned}	
	    \end{equation}
	    We can remove the weight function $V_1$ by the Mellin inversion. Thus \eqref{bound} equals
	    \begin{equation}\label{mellin}
	    	\begin{aligned}
	    		\Bigg|\frac{T^{1/2+\varepsilon}}{(2\pi i)^2}\int_{(0)}\int_{(0)}\sum_{\substack{N \leq T^{1/2+\varepsilon} \\
	    				N\,d y a d i c}} \sum_{n \geq 1}  \frac{\lambda_{f}(n)}{n^{1/2+\varepsilon}}{\left(\frac{n}{N}\right)}^{-s_1}\sum_{\substack{K \leq T^{1/2+\varepsilon} \\
	    				K\,d y a d i c}} \sum_{k\geq 1}\frac{1}{k^{1+\varepsilon}}\left(\frac{k}{K}\right)^{-s_2}\widetilde {V_{1}}(s_1)\widetilde{V_{1}}(s_2)\mathrm{~d} s_1\mathrm{~d} s_2\Bigg|.
	    	\end{aligned}	
	    \end{equation}
	    Using Lemma \ref{5}, we have that the upper bound of the sum of $n$ and $k$ in \eqref{mellin} is $O(1)$. Integrating by parts several times, we have that
	    $$
	    \begin{aligned}
	    	\widetilde {V_{1}}(s)=-\frac{1}{s}\int_{0}^{\infty}V_{1}^{\prime}(x)x^{s}\mathrm{~d} x\ll\frac{1}{(1+|s|)^A}
	    \end{aligned}
	    $$
	    for any $A\geq 1$. This allows us to reduce the $s$-integral to $O(1)$. Hence \eqref{bound} is bounded in a standard way by $O(T^{1/2+\varepsilon})$, as established in Lemma \ref{5}. Since the condition $N\leq T^{1/2+\varepsilon}$ can be removed with an acceptable error term of $O(T^{-2024})$, the $n$-sum within the integral simplifies $L(1+v,f)$. Similarly, the $k$-sum within the integral evaluates $\zeta(1+v+w)$. We now move the line of integration over $v$ to $\operatorname{Re}(v)=-1+\varepsilon$. In doing so we encounter simple poles at $v=0$ and $v=-w$ whose residues we next calculate. The integral on $\operatorname{Re}(v)=-1+\varepsilon$, $\operatorname{Re}(w)=\varepsilon$ contributes $O(T^{1/2+\varepsilon})$. The contribution from the residue at $v=-w$ is
		$$
		\begin{aligned}
			\frac{T}{2\pi i}\int_{\mathbb{R}}V(\xi) \mathrm{~d}\xi\int_{\varepsilon-iT^\varepsilon}^{\varepsilon+iT^\varepsilon}\frac{L(1-w,f)}{-w^2}e^{i\pi w/2}G(-w)G(w)\mathrm{~d} w.
		\end{aligned}
		$$		
		Finally we consider the contribution of the residue at $v=0$, namely
		$$
		\begin{aligned}
			\frac{T}{2\pi i}\int_{\mathbb{R}}V(\xi)\int_{\varepsilon-iT^\varepsilon}^{\varepsilon+iT^\varepsilon}\frac{L(1,f)\zeta(1+w)}{w}\left(\frac{T\xi}{2\pi}\right)^{w/2}e^{i\pi   w/4}G(w) \mathrm{~d} w\mathrm{~d} \xi.
		\end{aligned}
		$$
		We now move the line of integration over $w$ to $\operatorname{Re}(w)=-1+\varepsilon$, encountering a double pole at $w=0$. The integral on the new line can be bounded in a standard way by $O(T^{1/2+\varepsilon})$, and the residue of the double pole at $w=0$  is easily seen to be
		$$
		\begin{aligned}
		TL(1,f)\int_{\mathbb{R}}V(\xi)\left[\frac{i\pi}{4}+\frac{1}{2}\log\left(\frac{T\xi}{2\pi}\right)\right]\mathrm{~d} \xi.
		\end{aligned}
	    $$

		For $m\neq nk$ in $I_1$, we note that $I_1$ has the following estimate:
		$$
		\begin{aligned}
			I_1\ll T^{\varepsilon } \sup _{u \in\left[\varepsilon-i T^{\varepsilon}, \varepsilon+i T^{\varepsilon}\right]} \sup _{\substack{N \leq T^{1+\varepsilon} \\ M \leq T^{1/2+\varepsilon} \\ K \leq T^{1/2+\varepsilon}}}|I_{a}|,
		\end{aligned}
		$$ 
		where
		$$
		\begin{aligned}
			I_{a}=\int_{\mathbb{R}}V\left(\frac{t}{T}\right)\sum_{n\geq1}\frac{\lambda_{f}(n)}{n^{1/2+it+u}}V_1\left(\frac{n}{N}\right)\sum_{m\geq1}\frac{1}{m^{1/2-it+u}}V_1\left(\frac{m}{M}\right)\sum_{k\geq1}\frac{1}{k^{1/2+it+u}}V_1\left(\frac{k}{K}\right)t^{2u}\mathrm{~d} t.
		\end{aligned}
		$$ 
		Now letting $t=T\xi$, and then merging $t^{2u}$ into $V$ to get $V_2$, we have
		$$
		\begin{aligned}
			I_{a}=T^{1+2\varepsilon}\sum_{n\geq1}\frac{\lambda_{f}(n)}{n^{1/2+u}}V_1\left(\frac{n}{N}\right)\sum_{m\geq1}\frac{1}{m^{1/2+u}}V_1\left(\frac{m}{M}\right)\sum_{k\geq1}\frac{1}{k^{1/2+u}}V_1\left(\frac{k}{K}\right)\int_{\mathbb{R}}\frac{V_2(\xi)}{T^{2\varepsilon}}e^{iT\xi\log \frac{m}{nk}}\mathrm{~d}\xi.
		\end{aligned}
		$$ 
		We first consider the $\xi$-integral above. Let
		$$
		\varphi(\xi)=T\xi\log \frac{m}{nk},
		$$
		then we have
		$$
		\begin{aligned}
			\varphi^{\prime}(\xi)&=T\log\left(\frac{m}{nk}\right)\\
			&\gg T\min\left\{\frac{|nk-m|}{m},1\right\}\\&\gg T^{1/2-\varepsilon},\\
		\end{aligned}
		$$
		$$
		\begin{aligned} 
			\varphi^{(j)}(\xi)=0\leq T^{1/2-\varepsilon}\quad(j\geq2).
		\end{aligned}
		$$
		Hence, the integral in $I_{a}$ satisfies condition (1) in Lemma \ref{6}, where $Z=1$, $Y=T^{1/2-\varepsilon}$ and $X=\Delta$. Thus we have 
		$$
		\begin{aligned}
			\int_{\mathbb{R}}\frac{V_2(\xi)}{T^{2\varepsilon}} e^{i\xi T\log \frac{m}{nk}}\mathrm{~d}\xi\ll T^{-2024}.
		\end{aligned}
		$$
		Furthermore we have $I_{a}=O(T^{-2024})$. Having said all of above, this completes the proof.
	\end{proof}
	
	\noindent \textbf{Lemma 4.2} Given any $\varepsilon>0$, we have
	$$
	\begin{aligned}
		I_2\coloneq\frac{1}{(2 \pi i)^3}\int_{\mathbb{R}}V\left(\frac{t}{T}\right)S_{1}S_{3}S_{6}\mathrm{~d} t=O(T^{1/2+\varepsilon}),
	\end{aligned}
	$$
	$$
	\begin{aligned}
		I_3\coloneq\frac{1}{(2 \pi i)^3}\int_{\mathbb{R}}V\left(\frac{t}{T}\right)S_{1}S_{4}S_{5}\mathrm{~d} t=O(T^{1/2+\varepsilon}).
	\end{aligned}
	$$
	\begin{proof}
		The proofs of $I_2$ and $I_3$ are similar, so we show only the former. We note that $I_2$ has the following estimate:
		$$
		\begin{aligned}
			I_2\ll T^{\varepsilon } \sup _{u \in\left[\varepsilon-i T^{\varepsilon}, \varepsilon+i T^{\varepsilon}\right]} \sup _{\substack{N \leq T^{1+\varepsilon} \\ M \leq T^{1/2+\varepsilon} \\ K \leq T^{1/2+\varepsilon}}}|I_{b}|,
		\end{aligned}
		$$ 
		where
		$$
		\begin{aligned}
			I_{b}=\int_{\mathbb{R}}V\left(\frac{t}{T}\right)\sum_{n\geq1}\frac{\lambda_{f}(n)}{n^{1/2+it+u}}V_1\left(\frac{n}{N}\right)\sum_{m\geq1}\frac{1}{m^{1/2-it+u}}V_1\left(\frac{m}{M}\right)\sum_{k\geq1}\frac{1}{k^{1/2-it+u}}V_1\left(\frac{k}{K}\right){\left(\frac{t}{2\pi e}\right)}^{-it}t^{2u}\mathrm{~d} t.
		\end{aligned}
		$$ 
		Letting $t=T\xi$, and then merging $t^{2u}$ into $V$ to get $V_2$, we have
		$$
		\begin{aligned}
			I_{b}=T\sum_{n\geq1}\frac{\lambda_{f}(n)}{n^{1/2+u}}V_1\left(\frac{n}{N}\right)\sum_{m\geq1}\frac{1}{m^{1/2+u}}V_1\left(\frac{m}{M}\right)\sum_{k\geq1}\frac{1}{k^{1/2+u}}V_1\left(\frac{k}{K}\right)\int_{\mathbb{R}}V_2(\xi)e^{-iT\xi\log\frac{T\xi n}{2\pi emk}}\mathrm{~d}\xi.
		\end{aligned}
		$$ 
		Now we deal with $\xi$-integral. Let 
		$$
		\begin{aligned}
			\varphi(\xi)=-T\xi(\log T\xi+\log\frac{n}{2\pi emk}).
		\end{aligned}
		$$
		Hence
		$$
		\begin{aligned}
			\varphi^{\prime}(\xi)=-T\log\frac{T\xi n}{2\pi mk},
		\end{aligned}
		$$
		$$
		\begin{aligned}
			\varphi^{\prime\prime}(\xi)=-\frac{T}{\xi},\quad 
			\varphi^{(j)}(\xi)\asymp T\:(j\geq2).
		\end{aligned}
		$$
		We define the part of $I_b$ that satisfies $nT>4\pi mk$ and $2nT\leq mk\pi$ as $I_{b1}$, and we write the rest $I_{b2}$. Since $\varphi^{\prime}(\xi)\gg T$, then the integral $I_{b1}$ satisfies condition $(1)$ in Lemma \ref{6}, where $Z=1$, $Y=T$ and $X=\Delta$. Thus we have 
		$$
		\begin{aligned}
			I_{b1}\ll T^{-2024}.
		\end{aligned}
		$$
		If $mk\pi /2< nT\leq 4\pi mk$, $|\varphi^{\prime}(\xi)|\leq T\log 4$. Now there is $\xi_0=\frac{2\pi mk}{Tn}$, such that $\varphi^{\prime}(\xi_0)=0$. Then the integral satisfies condition $(2)$ in Lemma \ref{6}. Thus we have
		$$
		\begin{aligned}
			\int_{\mathbb{R}}V_2(\xi) e^{i\varphi(\xi)}\mathrm{~d}\xi=e\left(\frac{mk}{n}\right)\frac{i\sqrt{2\pi mk}}{T\sqrt{n}}V_3(\xi_0)+O(T^{-2024}),
		\end{aligned}
		$$
		where $V_3(\xi)$ is a fixed smooth function supported on $\xi\asymp 1$. We note that $I_{b2}$ is bounded by
		\begin{equation}\label{Ib2bound}
		\sup _{\substack{N \ll T^{1+\varepsilon}  \\ M \ll T^{1/2+\varepsilon}\\ K \ll T^{1/2+\varepsilon}}} \sup _{M K\asymp TN}\left| \sum_{n \geq 1}\lambda_{f}(n) \frac{1}{n^{1+u}} V_2\left(\frac{n}{N}\right)\sum_{m \geq 1}\frac{1}{m^{u}} V_2\left(\frac{m}{M}\right)\sum_{k \geq 1}\frac{e\left(\frac{mk}{n}\right)}{k^{u}} V_2\left(\frac{k}{K}\right)V_3\left(\frac{2\pi mk}{Tn}\right)\right|+O\left(T^{-2024}\right),
		\end{equation}
	    Since $MK\asymp TN$, we can truncate the $n$-sum above at $n\leq T^{2\varepsilon}$.\\
	    \textbf{Case 1:} $n\nmid m$, we have
	    $$
	    \begin{aligned}
	    \sum_{k\leq T^{1/2+\varepsilon}}e\left(\frac{mk}{n}\right)\leq\operatorname{min}\left\{T^{1/2+\varepsilon},{\left\|\frac{m}{n}\right\|}^{-1}\right\}\leq n\leq T^{2\varepsilon},
	    \end{aligned}
        $$
        where $\|\alpha\|$ denotes the distance of $\alpha$ to the nearest integer. Then by partial summation, we get the $k$-sum is bounded by $T^{\varepsilon}$. Applying Lemma \ref{3}, we get that 
        $$
        \begin{aligned}
        I_{b2}\ll T^{1/2+\varepsilon}.
        \end{aligned}
        $$
        \textbf{Case 2:} $n\mid m$, which means $e\left(\frac{mk}{n}\right)=1$. For the sum over $m$ and $k$ in \eqref{Ib2bound}, we apply the method of exponent pairs with A-process (see for example \cite[Chapter 3]{graham1991van}), by taking the exponent pair $(p,q)$ as
        $$
        \begin{aligned}
        (p,q)=\left(\frac{13}{84},\frac{55}{84}\right),
        \end{aligned}
        $$
        where $(p,q)$ according to Bourgain \cite{bourgain2017decoupling}. Hence we obtain
        $$
		\begin{aligned}
				I_{b2} & \ll \sup _{\substack{N \ll T^{2\varepsilon} \\
						M \ll T^{1/2+\varepsilon}\\K \ll T^{1/2+\varepsilon}}} \sup _{M K \asymp TN}\left|{\left(\frac{1}{M}\right)}^{p}M^{q} \cdot {\left(\frac{1}{K}\right)}^{p}K^{q}\right|+O\left(T^{-2024}\right) \\
				& \ll \sup _{\substack{N \ll T^{2\varepsilon} \\
						M \ll T^{1/2+\varepsilon}\\K \ll T^{1/2+\varepsilon}}} \sup _{M K \asymp TN}\left|(MK)^{1/2} \right|+O\left(T^{-2024}\right) \\
				& \ll T^{1/2+\varepsilon}.
		\end{aligned}
		$$
		Now using the estimates of $I_{b1}$ and $I_{b2}$ we obtain the lemma.
	\end{proof}
	
	\noindent \textbf{Lemma 4.3} We have
	$$
	\begin{aligned}
	 	I_4\coloneq\frac{1}{(2 \pi i)^3}\int_{\mathbb{R}}V\left(\frac{t}{T}\right)S_{1}S_{4}S_{6}\mathrm{~d} t=O(T^{-2024}),
	\end{aligned}
	$$
	$$
	\begin{aligned}
		I_5\coloneq\frac{1}{(2 \pi i)^3}\int_{\mathbb{R}}V\left(\frac{t}{T}\right)S_{2}S_{3}S_{5}\mathrm{~d} t=O(T^{-2024}),
	\end{aligned}
	$$
	$$
	\begin{aligned}
		I_6\coloneq\frac{1}{(2 \pi i)^3}\int_{\mathbb{R}}V\left(\frac{t}{T}\right)S_{2}S_{3}S_{6}\mathrm{~d} t=O(T^{-2024}),
	\end{aligned}
	$$
	$$
	\begin{aligned}
		I_8\coloneq\frac{1}{(2 \pi i)^3}\int_{\mathbb{R}}V\left(\frac{t}{T}\right)S_{2}S_{4}S_{6}\mathrm{~d} t=O(T^{-2024}).
	\end{aligned}
	$$
	\begin{proof}
	The above results are direct applications of the lemma \ref{6}, and we will only discuss the proof of $I_5$. We note that $I_5$ has the following estimate:
	$$
	\begin{aligned}
		I_5\ll T^{\varepsilon } \sup _{u \in\left[\varepsilon-i T^{\varepsilon}, \varepsilon+i T^{\varepsilon}\right]} \sup _{\substack{N \leq T^{1+\varepsilon} \\ M \leq T^{1/2+\varepsilon} \\ K \leq T^{1/2+\varepsilon}}}|I_{c}|,
	\end{aligned}
	$$ 
	where
	$$
	\begin{aligned}
		I_{c}=\int_{\mathbb{R}}V\left(\frac{t}{T}\right)\sum_{n\geq1}\frac{\lambda_{f}(n)}{n^{1/2-it+u}}V_1\left(\frac{n}{N}\right)\sum_{m\geq1}\frac{1}{m^{1/2-it+u}}V_1\left(\frac{m}{M}\right)\sum_{k\geq1}\frac{1}{k^{1/2+it+u}}V_1\left(\frac{k}{K}\right){\left(\frac{t}{2\pi e}\right)}^{-2it}t^{2u}\mathrm{~d} t.
	\end{aligned}
	$$ 
	Letting $t=T\xi$, and then merging $t^{2u}$ into $V$ to get $V_2$, we have
	$$
	\begin{aligned}
		I_{c}=T\sum_{n\geq1}\frac{\lambda_{f}(n)}{n^{1/2+u}}V_1\left(\frac{n}{N}\right)\sum_{m\geq1}\frac{1}{m^{1/2+u}}V_1\left(\frac{m}{M}\right)\sum_{k\geq1}\frac{1}{k^{1/2+u}}V_1\left(\frac{k}{K}\right)\int_{\mathbb{R}}V_2(\xi)e^{-iT\xi\log\frac{T^{2}{\xi}^{2}k }{4{\pi}^{2}e^{2}mn}}\mathrm{~d}\xi.
	\end{aligned}
	$$ 
	Now we deal with $\xi$-integral. Let
	$$
	\varphi(\xi)=-T\xi\log\frac{T^{2}{\xi}^{2}k }{4{\pi}^{2}e^{2}mn}.
	$$
	Then we have
	$$
	\begin{aligned}
		\varphi^{\prime}(\xi)=-T\log\frac{T^{2}{\xi}^{2}k }{4{\pi}^{2}mn}\asymp T\log T,
	\end{aligned}
	$$
	$$
	\begin{aligned}
		\varphi^{\prime\prime}(\xi)=-2\frac{T}{\xi},\quad 
		\varphi^{(j)}(\xi)\asymp T\:(j\geq2).
	\end{aligned}
	$$
	Hence, the integral in $I_{c}$ satisfies condition (1) in Lemma \ref{6}, where $Z=1$, $Y=T^{1+\varepsilon}$ and $X=\Delta$. Thus we have 
	$$
	\begin{aligned}
		\int_{\mathbb{R}}V_2(\xi)e^{-iT\xi\log\frac{T^{2}{\xi}^{2}k }{4{\pi}^{2}e^{2}mn}}\mathrm{~d}\xi\ll T^{-2024}.
	\end{aligned}
	$$
	Furthermore we have $I_{c}=O(T^{-2024})$. This completes the proof. 
	
	\end{proof}
	\noindent \textbf{Lemma 4.4} Given any $\varepsilon>0$, we have
	$$
	\begin{aligned}
		I_7\coloneq\frac{1}{(2 \pi i)^3}\int_{\mathbb{R}}V\left(\frac{t}{T}\right)S_{2}S_{4}S_{5 }\mathrm{~d} t=O(T^{1/2+\varepsilon}).
	\end{aligned}
	$$
	\begin{proof}
		We note that $I_7$ has the following estimate:
		$$
		\begin{aligned}
			I_7\ll T^{\varepsilon } \sup _{u \in\left[\varepsilon-i T^{\varepsilon}, \varepsilon+i T^{\varepsilon}\right]} \sup _{\substack{N \leq T^{1+\varepsilon} \\ M \leq T^{1/2+\varepsilon}\\ K \leq T^{1/2+\varepsilon}}}|I_{d}|,
		\end{aligned}
		$$ 
		where
		$$
		\begin{aligned}
			I_{d}=\int_{\mathbb{R}}V\left(\frac{t}{T}\right)\sum_{n\geq1}\frac{\lambda_{f}(n)}{n^{1/2-it+u}}V_1\left(\frac{n}{N}\right)\sum_{m\geq1}\frac{1}{m^{1/2+it+u}}V_1\left(\frac{m}{M}\right)\sum_{k\geq1}\frac{1}{k^{1/2+it+u}}V_1\left(\frac{k}{K}\right){\left(\frac{t}{2\pi e}\right)}^{-it}t^{2u}\mathrm{~d} t.
		\end{aligned}
		$$ 
		Letting $t=T\xi$, and merging $t^{2u}$ into $V$ to get $V_2$, we have
		$$
		\begin{aligned}
			I_{d}=T\sum_{n\geq1}\frac{\lambda_{f}(n)}{n^{1/2+u}}V_2\left(\frac{n}{N}\right)\sum_{m\geq1}\frac{1}{m^{1/2+u}}V_1\left(\frac{m}{M}\right)\sum_{k\geq1}\frac{1}{k^{1/2+u}}V_1\left(\frac{k}{K}\right)\int_{\mathbb{R}}V_2(\xi)e^{-iT\xi\log\frac{T\xi km }{2\pi en}}\mathrm{~d}\xi.
 		\end{aligned}
		$$
		Let
		$$\varphi(\xi)=-T\xi\log\frac{T\xi km }{2\pi en}.
		$$
		Then we have
		$$
		\begin{aligned}
			\varphi^{\prime}(\xi)=-T\log \frac{T\xi km}{2\pi n},
		\end{aligned}
		$$
		$$
		\begin{aligned} 
			\varphi^{\prime\prime}(\xi)=-\frac{T}{\xi},\quad
			\varphi^{(j)}(\xi)\asymp T\:(j\geq2).
		\end{aligned}
		$$
		We define the part of $I_d$ that satisfies $kmT>4\pi n$ and $2kmT\leq n\pi$ as $I_{d1}$, and we write the rest $I_{d2}$. Since $\varphi^{\prime}(\xi)\gg T$, then the integral $I_{d1}$ satisfies condition $(1)$ in Lemma \ref{6}, where $Z=1$, $Y=T$ and $X=\Delta$. Thus we have 
		$$
		\begin{aligned}
			I_{d1}\ll T^{-2024}.
		\end{aligned}
		$$
		Otherwise, $|\varphi^{\prime}(\xi)|\leq T\log 4$. Now there is $\xi_0=\frac{2\pi n}{Tkm}$, such that $\varphi^{\prime}(\xi_0)=0$. Then the integral satisfies condition $(2)$ in Lemma \ref{6}. Thus we have
		$$
		\begin{aligned}
			\int_{\mathbb{R}}V_2(\xi) e^{i\varphi(\xi)}\mathrm{~d}\xi=e\left(\frac{ n}{km}\right)\frac{i\sqrt{2\pi n}}{T\sqrt{km}}V_3(\xi_0)+O(T^{-2024}),
		\end{aligned}
		$$
		where $V_3(\xi)$ is a fixed smooth function supported on $\xi\asymp 1$. Then we have that 
		$$
		\begin{aligned}
			I_{d2}&=\sqrt{2\pi}i\sum_{kmT\asymp n}\frac{\lambda_{f}(n)}{n^{u}}\frac{e\left(\frac{n}{km}\right)}{(km)^{1+u}}V_2\left(\frac{n}{N}\right)V_2\left(\frac{m}{M}\right)V_2\left(\frac{k}{K}\right)V_3(\frac{2\pi n}{Tkm})+O(T^{-2024}).
		\end{aligned}
		$$
		By applying the Poisson summation and Lemma \ref{5}, we arrive at
		$$
		\begin{aligned}
			I_{d2}\ll T^{1/2+\varepsilon}+O(T^{-2024}).
		\end{aligned}
		$$
		Combining the upper bound of $I_{d1}$ and $I_{d2}$, we complete the proof.
	\end{proof}
	\noindent \textbf{Lemma 4.5} Given any $\varepsilon>0$, we have
	$$
	\begin{aligned}
			R\coloneq\frac{1}{(2 \pi i)^3}\int_{\mathbb{R}}V\left(\frac{t}{T}\right)S_{7}\mathrm{~d} t=O(T^{1/2+\varepsilon}).
	\end{aligned}
	$$
	\begin{proof}
		We can convert $R$ into the following form:
		$$
		\begin{aligned}
			R\ll T^{\varepsilon } \sup _{u \in\left[\varepsilon-i T^{\varepsilon}, \varepsilon+i T^{\varepsilon}\right]} \sup _{\substack{N \leq T^{1+\varepsilon} \\ M \leq T^{1/2+\varepsilon}\\ K \leq T^{1/2+\varepsilon}}}|R_{1}+R_{2}+R_{3}+R_{4}+R_{5}|,
		\end{aligned}
		$$
		where 
		$$
		\begin{aligned}
			&R_{1}=T^{-3/4}\int_{\mathbb{R}}V\left(\frac{t}{T}\right)\sum_{n\geq1}\frac{\lambda_{f}(n)}{n^{1/2+it+u}}V_1\left(\frac{n}{N}\right)\sum_{m\geq1}\frac{1}{m^{1/2-it+u}}V_1\left(\frac{m}{M}\right)t^{3u/2}\mathrm{~d} t,\\
			&R_{2}=T^{-1/2}\int_{\mathbb{R}}V\left(\frac{t}{T}\right)\sum_{m\geq1}\frac{1}{m^{1/2-it+u}}V_1\left(\frac{m}{M}\right)\sum_{k\geq1}\frac{1}{k^{1/2+it+u}}V_1\left(\frac{k}{K}\right)t^{u}\mathrm{~d} t,\\ 
			&R_{3}=T^{-3/2}\int_{\mathbb{R}}V\left(\frac{t}{T}\right)\sum_{n\geq1}\frac{\lambda_{f}(n)}{n^{1/2+it+u}}V_1\left(\frac{n}{N}\right)t^{u}\mathrm{~d} t,\\
			&R_{4}=T^{-5/4}\int_{\mathbb{R}}V\left(\frac{t}{T}\right)\sum_{m\geq1}\frac{1}{m^{1/2-it+u}}V_1\left(\frac{m}{M}\right)t^{u/2}\mathrm{~d} t,\\
			&R_{5}=T^{-2}\int_{\mathbb{R}}V\left(\frac{t}{T}\right)\mathrm{~d} t.
		\end{aligned}
		$$
		For $R_{1}$, by Cauchy-Schwarz inequality, Lemma \ref{3} and Lemma \ref{7}, we have
		$$ 
		\begin{aligned}
			R_{1}&\ll T^{-3/4+\varepsilon}\left(\int_{T}^{2T}\left|\sum_{n\geq1}V_1\left(\frac{n}{N}\right)\frac{\lambda_{f}(n)}{n^{1/2+it+u}}\right|^{2}\mathrm{~d} t\right)^{1/2}\left(\int_{T}^{2T}\left|\sum_{m\geq1}V_1\left(\frac{m}{M}\right)\frac{1}{m^{1/2-it+u}}\right|^{2}\mathrm{~d} t\right)^{1/2}\\
			&\ll T^{-3/4+\varepsilon}\left[(T+N)\sum_{n\geq1}V_1\left(\frac{n}{N}\right)\frac{|\lambda_{f}(n)|^{2}}{n^{1+2\varepsilon}}(T+M)\sum_{m\geq1}V_1\left(\frac{m}{M}\right)\frac{1}{m^{1+2\varepsilon}}\right]^{1/2}\\
			&\ll T^{1/4+\varepsilon}.
		\end{aligned}
		$$ 
		The rest of the proof is in complete agreement with the proof of $R_1$, hence it suffices to estimate $R_1$. Finally we can obtain that the upper bound of $R$ is $O(T^{1/2+\varepsilon})$.  
	\end{proof}
	
	Theorem \ref{t1} follows upon combining the previous lemmas.
	
	\section{Proof of Corollaries}
	\subsection{Proof of Corollary \ref{c3}}
	Note that from \eqref{Ivic}, we have
	$$
	\begin{aligned}
		\int_{T}^{T+T/\Delta}\left|\zeta\left(\frac{1}{2}-it\right)\right|^4\mathrm{~d} t\ll \frac{T^{1+\varepsilon}}{\Delta}+T^{2/3+\varepsilon}.
	\end{aligned}
    $$
    In particular, if $\Delta\leq T^{1/3}$, we have
    $$
    \begin{aligned}
    	\int_{T}^{T+T/\Delta}\left|\zeta\left(\frac{1}{2}-it\right)\right|^4\mathrm{~d} t\ll \frac{T^{1+\varepsilon}}{\Delta};
    \end{aligned}
    $$
    see also \cite{iwaniec1980fourier}. Similarly, under the same assumption $\Delta\leq T^{1/3}$ and from \eqref{good}, we have
    $$
    \begin{aligned}
		\int_{T}^{T+T/\Delta}\left|L\left(\frac{1}{2}+it, f\right)\right|^2\mathrm{~d} t\ll\frac{T^{1+\varepsilon}}{\Delta}.
	\end{aligned}
	$$
	Similar to \cite[section 6]{lin2021analytic}, we choose the smooth function $V$ to be supported on $[1,2]$ and $V(x)=1$ on $[1+1/\Delta,2-1/\Delta]$. Then Theorem \ref{t1} yields
	$$
	\begin{aligned}
		\int_{T}^{2T}L\left(\frac{1}{2}+it, f\right){\left|\zeta\left(\frac{1}{2}-it\right)\right|}^2\mathrm{~d} t&=\frac{c}{2}L(1,f)T\log T+c_{f}T +O(T^{\frac{1}{2}+\varepsilon})\\&+\int_{T}^{T+T/\Delta}\left(1-V\left(\frac{t}{T}\right)\right)L\left(\frac{1}{2}+it, f\right){\left|\zeta\left(\frac{1}{2}-it\right)\right|}^2\mathrm{~d} t\\&+\int_{2T-T/\Delta}^{2T}\left(1-V\left(\frac{t}{T}\right)\right)L\left(\frac{1}{2}+it, f\right){\left|\zeta\left(\frac{1}{2}-it\right)\right|}^2\mathrm{~d} t.
	\end{aligned}
	$$
	From Cauchy-Schwarz inequality, this implies that
	$$
	\begin{aligned}
		\int_{T}^{2T}L\left(\frac{1}{2}+it, f\right){\left|\zeta\left(\frac{1}{2}-it\right)\right|}^2\mathrm{~d} t=\frac{c}{2}L(1,f)T\log T+c_{f}T +O(T^{\frac{1}{2}+\varepsilon})+O(T^{1+\varepsilon}/\Delta) .
	\end{aligned}
	$$
	Corollary \ref{c3} then follows by choosing $\Delta=T^{1/3}$.
	
	\subsection{Proof of Corollary \ref{c4}}
	Corollary \ref{c4} is obtained in the same way as above section.
	
	\subsection*{Acknowledgements}The author thanks Prof. Yongxiao Lin for proposing the project and for his guidance. She likes to thank Prof. Bingrong Huang for pointing an error in the first version of the paper. The author was supported by the National Key R\&D Program of China (No.2021YFA1000700). 
	
	\bibliographystyle{abbrv}
	\bibliography{ref}

\begin{thebibliography}{10}

\bibitem{bourgain2017decoupling}
J.~Bourgain.
\newblock Decoupling exponential sums and the {R}iemann zeta function.
\newblock {\em Journal of the American Mathematical Society}, 30(1):205--224,
2017.

\bibitem{bourgain2018decoupling}
J.~Bourgain and N.~Watt.
\newblock Decoupling for perturbed cones and the mean square of $|\zeta(\frac{1}{2}+it)|$.
\newblock {\em International Mathematics Research Notices},
2018(17):5219--5296, 2018.

\bibitem{das2015simultaneous}
S. Das and R.~R. Khan. 
\newblock Simultaneous nonvanishing of Dirichlet $L$-functions and twists of Hecke-Maass $L$-functions.
\newblock {\em J. Ramanujan Math. Soc},
30(3):237--250, 2015.

\bibitem{duke2002subconvexity}
W.~Duke, J.~B. Friedlander, and H.~Iwaniec.
\newblock The subconvexity problem for {A}rtin {$L$}-functions.
\newblock {\em Inventiones mathematicae}, 149(3):489--577, 2002.

\bibitem{good1982square}
A.~Good.
\newblock The square mean of {D}irichlet series associated with cusp forms.
\newblock {\em Mathematika}, 29(2):278--295, 1982.

\bibitem{graham1991van}
S.~W. Graham and G.~Kolesnik.
\newblock Van der Corput's method of exponential sums, volume 126.
\newblock {\em Cambridge University Press}, 1991.

\bibitem{hafner1989sums}
J.~L. Hafner and A.~Ivic.
\newblock On sums of {F}ourier coefficients of cusp forms.
\newblock {\em Enseign. Math}, 35(2):375--382, 1989.

\bibitem{hardy1916contributions}
G.~H. Hardy and J.~E. Littlewood.
\newblock Contributions to the theory of the {R}iemann zeta-function and the
theory of the distribution of primes.
\newblock {\em Acta Mathematica}, 41(1):119--196, 1916.

\bibitem{heath1979fourth}
D.~R. HeathBrown.
\newblock The fourth power moment of the Riemann zeta function.
\newblock {\em Proceedings of the London Mathematical Society}, 3(3):385--422, 1979.

\bibitem{huang2021rankin}
B.~Huang.
\newblock On the {R}ankin--{S}elberg problem.
\newblock {\em Mathematische Annalen}, 381(3):1217--1251, 2021.

\bibitem{ingham1928mean}
A.~E. Ingham.
\newblock Mean-value theorems in the theory of the {R}iemann zeta-function.
\newblock {\em Proceedings of the London Mathematical Society}, 2(1):273--300,
1928.

\bibitem{ivic1995fourth}
A.~Ivi{\'c}.
\newblock On the fourth moment of the {R}iemann zeta functions.
\newblock {\em Publications de l'Institut Math{\'e}matique}, 57(77):101--110,
1995.

\bibitem{ivic1995fourth2}
A.~Ivi{\'c} and Y.~Motohashi.
\newblock On the fourth power moment of the {R}iemann zeta-function.
\newblock {\em Journal of Number Theory}, 51(1):16--45, 1995.

\bibitem{iwaniec1980fourier}
H.~Iwaniec.
\newblock {F}ourier coefficients of cusp forms and the {R}iemann zeta-function.
\newblock {\em Seminaire de Th{\'e}orie des Nombres de Bordeaux}, 1--36, 1980.

\bibitem{iwaniec2021spectral}
H.~Iwaniec.
\newblock Spectral methods of automorphic forms, volume~53.
\newblock {\em American Mathematical Society}, Revista Matem{\'{a}}tica Iberoamericana (RMI), Madrid, Spain, 2002.

\bibitem{iwaniec2021analytic}
H.~Iwaniec and E.~Kowalski.
\newblock Analytic number theory, volume~53.
\newblock {\em American Mathematical Society}, 2004.

\bibitem{keating2000random}
J.~P. Keating and N.~C. Snaith.
\newblock Random matrix theory and $\zeta(1/2+ it)$.
\newblock {\em Communications in Mathematical Physics}, 214(1):57--89, 2000.

\bibitem{kiral2019oscillatory}
E.~M. K{\i}ral, I.~Petrow, and M.~P. Young.
\newblock Oscillatory integrals with uniformity in parameters.
\newblock {\em Journal de th{\'e}orie des nombres de Bordeaux}, 31(1):145--159,
2019.

\bibitem{lin2021analytic}
Y.~Lin and Q.~Sun.
\newblock Analytic twists of {$\rm{GL_3}\times \rm{GL_2}$} automorphic forms.
\newblock {\em International Mathematics Research Notices}, 2021(19):15143--15208,
2021.

\bibitem{meurman1987order}
T.~Meurman.
\newblock On the order of the {M}aass {$L$}-function on the critical line, volume~53.
\newblock {\em Number theory},325--354,
1987.

\bibitem{potter1940mean}
H.~Potter.
\newblock The mean values of certain {D}irichlet series, \uppercase\expandafter{\romannumeral 1}.
\newblock {\em Proceedings of the London Mathematical Society}, 2(1):467--478,
1940.

\bibitem{MR1683661}
N.~I. Zavorotny{\u{\i}}.
\newblock On the fourth moment of the {R}iemann zeta function.
\newblock {\em Automorphic functions and number theory, {P}art {\uppercase\expandafter{\romannumeral 1}}, {\uppercase\expandafter{\romannumeral 2}}
({R}ussian)}, 69--124a, 254. Akad. Nauk SSSR, Dal' nevostochn.
Otdel., Vladivostok, 1989.

\end{thebibliography}
\end{document}